\documentclass{amsart}
\usepackage{amsmath,amssymb,enumerate,graphicx,hyperref,mathrsfs,verbatim,xypic}

\newcommand{\Z}{{\mathbb{Z}}}

\newcommand{\R}{{\mathbb{R}}}

\newtheorem{thm}{Theorem} 

\newtheorem*{spec}{Theorem 1$'$}

\begin{document}

\title{Realizing homology classes up to cobordism}



\author{Mark Grant}
\address{Institute of Mathematics, University of Aberdeen, Fraser Noble Building, Aberdeen AB24 3UE}
\email{mark.grant@abdn.ac.uk}

\author{Andr\'as Sz\H ucs}
\address{ELTE Analysis Department, 1117 Budapest, P\'azm\'any P\'eter s\'et\'any 1/C, Hungary}
\email{szucs@math.elte.hu}

\author{Tam\'as Terpai}
\address{Alfr\'ed R\'enyi Institute of Mathematics, 1053 Budapest, Re\'altanoda u. 13-15., Hungary}
\email{terpai@math.elte.hu}

\begin{abstract}
It is known that neither immersions nor maps with a fixed finite set of multisingularities are enough to realize all mod $2$ homology classes in manifolds. In this paper we define the notion of realizing a homology class up to cobordism; it is shown that for realization in this weaker sense immersions are sufficient, but maps with a fixed finite set of multisingularities are still insufficient.
\end{abstract}


\maketitle

\section{Introduction}

In $1949$ Steenrod \cite{Eilenberg} posed the following question: given a homology class $h$ of a space $X$, does there exist a closed manifold $V$ and a continuous map $f: V \to X$ such that $f_*[V]=h$, where $[V]$ is the fundamental class of $V$? Thom's famous result answers the question affirmatively if $h$ is a $\Z_2$-homology class, and shows that for integral homology the answer in general is negative. It is a natural further question whether $f$ can be chosen to be ``nice'' if $X$ itself is a smooth manifold. For example, can it be always an embedding or an immersion? If not, then can $f$ be chosen to have only mild singularities?

For embeddings Thom himself gave some necessary and sufficient conditions. From these conditions it is not hard to deduce that there are $\Z_2$-homology classes of codimension $2$ not realizable by embeddings in some manifolds.

In \cite{GrantSzucs} it was shown that for any $k>1$ there is a manifold $M$ (of dimension approximately $4k$) and a cohomology class $\alpha \in H^k(M;\Z_2)$ such that the Poincar\'e dual of $\alpha$ cannot be realized by an immersion. Moreover it was shown there that for any $k>1$ singular maps of finite complexity (see Section \ref{sec:singular} for the precise definition)  are insufficient to realize all codimension $k$ homology classes in manifolds.

Therefore in order to obtain positive answers it is natural to relax the notion of ``realization of a homology class''. The relaxed version we use will be ``realization up to cobordism''. For this purpose we define the cobordism group of pairs $(M^n,\alpha)$ where $M^n$ is a closed smooth $n$-manifold and $\alpha \in H^k(M;\Z_2)$ for a fixed $k$.

\par
{\sc Definition:} Given two pairs $(M^n,\alpha)$ and $(N^n,\beta)$ we say that they are \emph{cobordant} if there is a pair $(W^{n+1},\gamma)$ such that $W^{n+1}$ is a compact $(n+1)$-manifold with boundary $\partial W = M \sqcup N$ and $\gamma \in H^k(W;\Z_2)$ is a cohomology class such that $\gamma|_M = \alpha$ and $\gamma|_N = \beta$.
\par

{\sc Remark:} The obtained group of pairs is clearly isomorphic to $\mathfrak N_n(K(\Z_2,k))$, the $n$th bordism group of the Eilenberg-MacLane space $K(\Z_2,k)$.

\par
{\sc Definition:} Let $\mathcal F$ be a class of smooth maps (for example, embeddings, immersions, or singular maps of some given complexity). We say that a pair $(M,\alpha)$ is \emph{$\mathcal F$-realizable} if there exist a closed manifold $V$ and a map $f:V \to M$ such that $f \in \mathcal F$ and $f_*[V]$ is Poincar\'e dual to $\alpha$. We say that $(M,\alpha)$ is \emph{$\mathcal F$-realizable up to cobordism} if there is an $\mathcal F$-realizable pair $(N,\beta)$ cobordant to $(M,\alpha)$.

\par
We show that this relaxation allows to give a positive answer in the case of immersions but for singular maps of finite complexity the answer remains negative.

\section{Realization by immersions}\label{sec:basic}

\begin{thm}\label{thm:main}
Any pair $(M,\alpha)$ is realizable by immersions up to cobordism.
\end{thm}

For conciseness, (co)homology coefficients $\Z_2$ will be omitted and $K$ will stand for $K(\Z_2,k)$.

In what follows, $MO(k)$ denotes as usual the Thom space of the universal vector bundle over $BO(k)$, and for any space $X$ we denote by $\Gamma X$ the space $\Omega^\infty S^\infty X = \lim_{N \to \infty} \Omega^N S^N X$. Recall that $\Gamma MO(k)$ is the classifying space of codimension $k$ immersions, in particular, the group of cobordism classes of codimension $k>0$ immersions into a fixed closed manifold $P$ (where cobordisms are codimension $k$ immersions into $P \times [0,1]$) is isomorphic to the group of homotopy classes $[P,\Gamma MO(k)]$.


It is well-known that $\Gamma MO(k)$ is stably equivalent to a bouquet that contains $MO(k)$ (i.e. there is a space $Y$ such that $\Gamma MO(k) \underset{stably}{\cong} MO(k) \vee Y$). Hence $H^*(MO(k))$ embeds naturally into $H^*(\Gamma MO(k))$. In particular the Thom class $u_k \in H^k(MO(k))$ can be considered (uniquely, since $Y$ is known to be $2k-1$-connected) as a cohomology class of $\Gamma MO(k)$. Denote by $u$ the corresponding map into $K$, that is, $u: \Gamma MO(k) \to K$ has the property that $u^*(\iota_k) = u_k$, where $\iota_kl \in H^k(K)$ is the fundamental class.

Alternatively, we may use the universal property of the functor $\Gamma$ that is as follows (\cite[p. 39.]{CLM}, \cite[pp.42--43.]{M}): for any map $f:X\to Y$ from a compactly generated Hausdorff space $X$ to an infinite loop space $Y$ there is a homotopically unique extension $\hat f: \Gamma X \to Y$ that is an infinite loop map. Applying this property to $u_k$ yields the map $u$.

For any $P$ the map $u^P_*: [P,\Gamma MO(k)] \to [P,K]$ induced by $u$ associates to (a cobordism class of) an immersion the Poincar\'e dual of the homology class represented by the immersion.

This shows that Theorem \ref{thm:main} has the following equivalent reformulation:

\begin{spec}
The map $u: \Gamma MO(k) \to K$ induces an epimorphism of the bordism groups in any dimension. That is, for any $n$
$$
u_*: \mathfrak N_n(\Gamma MO(k)) \to \mathfrak N_n(K)
$$
is onto.
\end{spec}

\begin{proof}
It is well-known (\cite{ConnerFloyd}) that there is an isomorphism $H_*(X; \Z_2) \otimes \mathfrak N_* \to \mathfrak N_*(X)$, natural in $X$, defined by taking a representative $[\hat \alpha : M_\alpha \to X] \in \mathfrak N_*(X)$ for all elements $\alpha$ of a basis of $H_*(X)$ and mapping $\sum_j \alpha_j \otimes [N_j]$ to $\sum_j [\hat \alpha_j \circ pr_j : M_{\alpha_j} \times N_j \to X]$, where $pr_j :  M_{\alpha_j} \times N_j \to M_{\alpha_j}$ is the projection. Hence a map induces epimorphism of the (unoriented) bordism groups if and only if it does so in the $\Z_2$-homology groups.

For later use, recall that for any space $X$ the ring $H_*(\Gamma X)$ is a polynomial ring (multiplication being the Pontryagin product) in variables $x_\lambda$, $y_{I,\lambda}$, where $\{ x_\lambda \}_\lambda$ is a homogeneous basis of $H_*(X)$ and $y_{I,\lambda}$ are further variables defined using Kudo-Araki operations as $y_{I,\lambda} = Q^I x_\lambda$ (their precise description will be unimportant in our argument).

In order to show that
$$
u_* : H_*(\Gamma MO(k)) \to H_*(K)
$$
is onto it is enough to show that the composition
$$
\varphi \overset{\operatorname{def}}{:} \overline H_*(MO(k)) \overset{(u_k)_*}{\to} \overline H_*(K) \overset{p}{\to} Q(H_*(K)) = \overline H_*(K)/\mu\left(\overline H_*(K) \otimes \overline H_*(K)\right)
$$
is onto, where $\mu: H_*(K) \otimes H_*(K) \to H_*(K)$ is the multiplication map and $p$ is the natural projection onto the quotient group of indecomposables. Indeed, assume that $\varphi$ is onto and for all $j$ choose elements in $H_j(K)$ such that they form a (linear) basis in $\overline H_j(K)/\mu\left(\overline H_j(K) \otimes \overline H_j(K) \right)$. It is easy to see by induction on $j$ that the chosen elements generate $\overline H_*(K)$ multiplicatively and hence the subring of $H_*(\Gamma MO(k))$ generated by the preimages of these elements is mapped onto the entire $H_*(K)$ (here we use that $u_*$ is a ring homomorphism, since $u$ is an infinite loop map).

Hence to prove Theorem \ref{thm:main} we have to show that $\varphi:H_*(MO(k)) \to QH_*(K)$ is onto. This is equivalent to the dual homomorphism $\varphi^*$ being injective. By \cite[Proposition 3.10.]{MM}, the dual of $QH_*(K)$ is $PH^*(K)$, the submodule of primitive elements of the Hopf algebra $H^*(K)$. This latter group is known to be
$$
PH^*(K) = \Z_2\left\langle Sq^I \iota_k : I \text{ admissible of excess } e(I) \leq k \right\rangle,
$$
the vector space over $\Z_2$ freely generated by the $Sq^I \iota_k$ (see eg. \cite[p. 23.]{Clement}). The dual of $H_*(MO(k))$ is $H^*(MO(k))$ and can be identified with the ideal generated by $w_k$ in $\Z_2[w_1,\dots,w_k]$ ($w_k$ corresponds to the Thom class $u_k$). The map $\varphi^*$ maps $\iota_k$ to $u_k$ and then to $w_k$, and commutes with the action of the Steenrod algebra, allowing to calculate the image of $\varphi^*$.
\par
Finally, we need to show that the set $\left\{ Sq^I(w_k) : I\text{ is admissible with } e(I)\leq k\right\}$ is linearly independent in the ideal $(w_k) \subset \Z_2[w_1,\dots,w_k]$. This is the immediate consequence of \cite[Remark 2.4.]{PW} that shows that the Steenrod algebra acts freely unstably on $w_k$ in $H^*(BO(k))$, and this finishes the proof of Theorem \ref{thm:main}.
\end{proof}

\section{Non-realizability up to cobordism by singular maps of finite complexity}\label{sec:singular}

Recall some definitions from singularity theory that are necessary for the formulation of Theorem \ref{thm:singular}.
\par
{\sc Definition:} Fix a natural number $k \geq 1$ and consider equivalence classes of germs $\eta: (\R^{n-k},0) \to (\R^n,0)$, $n\geq k$, up to left-right equivalence and stabilization, that is, we consider $\eta$ to be equivalent to $\eta \times id_{\R^1} : (\R^{n-k+1},0) \to (\R^{n+1},0)$. An equivalence class is called a (codimension $k$) \emph{local singularity} (even if its rank is maximal).
\par
{\sc Definition:} A \emph{multisingularity} is a finite multiset (set with elements equipped with multiplicities) of local singularities.
\par
{\sc Definition:} Let $f: M\to N$ be a smooth map such that for any $y\in N$ the preimage $f^{-1}(y)$ is a finite set. For $y\in N$ and $f^{-1}(y) = \left\{x_1, \dots,x_m\right\}$ let $\left[f_{x_j}\right]$ denote the local singularity class of the germ $f$ at $x_j$. The multiset $\left\{ \left[f_{x_1}\right],\dots,\left[f_{x_m}\right] \right\}$ is called the \emph{multisingularity of $f$ at $y$}.
\par
{\sc Definition:} Let $\tau$ be a set of multisingularities. The map $f$ is said to be a \emph{$\tau$-map} if its multisingularity at any point $y\in N$ belongs to $\tau$.

\begin{thm}\label{thm:singular}
Let $\tau$ be any finite set of multisingularities of codimension $k>1$ stable maps and let $\mathcal F$ be the class of $\tau$-maps. Then the class $\mathcal F$ is insufficient for realizing up to cobordism all codimension $k$ homology classes in manifolds. That is, for any $k>1$ there is a pair $(M,\alpha)$ with $M$ a smooth manifold and $\alpha \in H^k(M)$ such that $(M,\alpha)$ is not $\mathcal F$-realizable up to cobordism.
\end{thm}

\begin{proof}
The proof given in \cite[Theorem 1.3.]{GrantSzucs} for non-realizability of homologies by $\tau$-maps also proves the stronger statement of Theorem \ref{thm:singular}. In that proof there was a classifying space $X_\tau$ for $\tau$-maps (analogously to $\Gamma MO(k)$ being the classifying space for immersions). $X_\tau$ has a single nonzero element in its first nontrivial (reduced) cohomology group, $H^k(X_\tau)$, which can be called the Thom class $u_\tau: X_\tau \to K$. If any pair $(M,\alpha)$ could be realizable up to cobordism by $\tau$-maps, then the map $u_\tau$ would induce an epimorphism $\left(u_\tau\right)_*:\mathfrak N_*(X_\tau) \to \mathfrak N_*(K)$ between the unoriented bordism groups or, equivalently, between the homology groups (using the same argument as in the proof of Theorem \ref{thm:main}). But \cite{GrantSzucs} shows that for any sufficiently high dimension $j$ (under the assumption that $k>1$) we have ${\dim_{\Z_2} H_j(X_\tau) < \dim_{\Z_2} H_j(K)}$, hence $\left(u_\tau \right)_*: H_j(X_\tau) \to H_j(K)$ cannot be surjective.
\end{proof}

{\sc Remark:} In particular, embeddings or immersions with self-intersection multiplicity bounded by a fixed number are insufficient for realizing all homology classes in manifolds even up to cobordism.

\end{document}